\documentclass[reqno,a4paper,11pt]{amsart}
\usepackage{hyperref}
\usepackage{tikz}
\usepackage[active]{srcltx}
\usepackage[curve]{xypic}
\usepackage{graphicx}
\usepackage{mathrsfs}

\usepackage{color}

\usepackage[active]{srcltx}
\usepackage{epsfig,amsmath}
\usepackage[english]{babel}
\usepackage{amsmath,amsfonts,amssymb,amscd,color,epsfig,amsthm}%eufrak
\newtheorem{newthm}{Theorem}
\newtheorem{theorem}{Theorem}[section]
\newtheorem{lemma}[theorem]{Lemma}
\newtheorem{proposition}[theorem]{Proposition}
\newtheorem{corollary}[theorem]{Corollary}

\newtheorem{definition}[theorem]{Definition}

\theoremstyle{remark}

\theoremstyle{plain}

\numberwithin{equation}{section}

\newcommand{\wt}{\widetilde}

\def\BBB{{\cal B}}

\def\FFF{{\cal F}}

\def\JJJ{{\cal J}}

\def\NNN{{\cal N}}

\def\UUU{{\cal U}}

\def\XXX{{\cal X}}
\def\YYY{{\cal Y}}

\def\g{\gamma}

\def\De{\Delta}
\def\de{\delta}

\def\R{\mbox{$\mathbb R$}}

\def\C{\mbox{$\mathbb C$}}

\def\D{\mathbb D}

\def\N{\mbox{$\mathbb N$}}

\def\lv{ \left(\begin{matrix} }
 \def\rv{\end{matrix}\right)}

\def\cal{\mathcal}

\def\dw{{\dw}}

\newcommand{\mylabel}[1]{\label{#1}}

\newcommand{\REFEQN}[1] { \begin{equation}\mylabel{#1} }
\newcommand{\ENDEQN}{\end{equation}}
\newcommand{\REFTHM}[1] { \begin{theorem}\mylabel{#1} }
\newcommand{\ENDTHM}{\end{theorem}}
\newcommand{\REFNTH}[1] { \begin{newthm}\mylabel{#1} }
\newcommand{\ENDNTH}{\end{newthm}}
\newcommand{\REFPROP}[1]{\begin{proposition}\mylabel{#1} }
\newcommand{\ENDPROP}{\end{proposition} }
\newcommand{\REFLEM}[1]{\begin{lemma}\mylabel{#1} }
\newcommand{\ENDLEM}{\end{lemma} }
\newcommand{\REFCOR}[1]{\begin{corollary}\mylabel{#1} }
\newcommand{\ENDCOR}{\end{corollary} }

\def\pf{postcritically-finite }

\def\ov{\overline}

\usepackage{tikz}
\usetikzlibrary{arrows}
\tikzstyle{every picture}=[> = to]
\tikzset{cdlabel/.style={execute at begin node=$\scriptstyle,execute at end node=$}}
\tikzset{implication/.style={double equal sign distance, -implies}}
\tikzset{biimplication/.style={double equal sign distance, implies-implies}}

\title{Density of hyperbolicity of real Newton maps}
\author{Yan Gao}%\footnote{The auther is supported by the grant no. 11501383 of NSFC.}}
\date{\today}
\begin{document}
\maketitle
\begin{abstract}
In this paper, based on the known rigidity theorems of Newton maps (\cite{DS}, \cite{RYZ}) and real polynomials (\cite{KSS1}, \cite{CS}), we prove the density of hyperbolicity in the family of real Newton maps of degree $d\geq3$ with all free critical points real.

\vspace{0.1cm}

\vspace{0.1cm}

\noindent{\bf AMS(2010) Subject Classification}: 37B40, 37F10, 37F20.
\end{abstract}
\section{Introduction}

Recall that a rational map of degree more than one is called \emph{hyperbolic} if  every critical point belongs to the
Fatou set and is attracted to an attracting cycle. One central conjecture in the complex dynamics is the following.

\noindent HYPERBOLICITY CONJECTURE. \emph{Each rational map can be approximated by hyperbolic rational maps with the same degree.}

We briefly recall the recent progress on this problem. In the real setting, Lyubich
\cite{L} and Graczyk-\'{S}wiatek \cite{GS} independently proved that any real quadratic polynomial can be approximated by real quadratic hyperbolic polynomials; in general real polynomial family, Kozlovski \emph{et al} solved this problem for real polynomial with real critical points in \cite{KSS1} and for general interval maps and
circle maps in \cite{KSS2}; in \cite{RvS}, Rempe-Gillen and van Strien proved the density of hyperbolicity in some spaces of real transcendental entire functions. In the complex case, Kozlovski-van Strien \cite{KvS} proved that any non-renormalizable polynomial which has only hyperbolic periodic points can be approximated by a hyperbolic polynomial of the same degree; while W. Peng \emph{et al} proved the density of hyperbolicity in the family of rational maps with Cantor Julia set \cite{PYZ}.  M. Aspenberg \cite{AR,AR2}  shows that any Misiurewicz rational map  can be approximated by hyperbolic rational maps.

  Each of works mentioned above, except the last one, relies on a so-called Rigidity Theorem in each of their corresponding families. In particular, the proof of density of hyperbolicity in the real polynomial family
is based on the following rigidity theorem (see \cite[Rigidity Theorem]{KSS1} and \cite[Theorem 1.1]{CS}).

\noindent{\bf Theorem A.}
\emph{Let $f,\wt{f}$ be two real polynomials of degree $d$ with all critical points real, and
 topologically conjugate on $\R$. Moreover, assume that the topological conjugacy is a bijection
between
\begin{itemize}
\item the sets of parabolic periodic points,
\item the sets of critical points,
\end{itemize}
and that the orders of corresponding critical points are the same. Then $f$ and $\wt{f}$ are quasi-conformally conjugate on $\C$.}

Recently,  Drach-Schleicher \cite{DS} and Roesch-Yin-Zeng \cite{RYZ} independently proved a rigidity result in the family of \emph{Newton maps}, i.e., the maps with the form
\[f_p(z)=z-\frac{p(z)}{p'(z)}\text{ for all $z\in\C$},\]
where $p(z)$ is a polynomial.
Such maps are well known because their iterations are used to find the
roots of the polynomial $p$: each root of $p$ is an attracting fixed point of $f_p$.

%So we call the points converging by the iteration of $f_p$ to the roots of $p$ the \emph{basin of roots} or simply \emph{roots-basin} of $f_p$. A critical point of $f_p$ is called \emph{free} if it is not in the roots-basin of $f_p$.

%We denote by $\NNN_d$ the space of Newton maps of degree $d$, and by $\NNN_d^{\rm pf}$ the subspace of $\NNN_d$ consisting of \pf ones.

% The cases $d<3$ are excluded because they are trivial. Observed that $f$ arises naturally when Newton¡¯s method is applied to find
%the roots of $P$. Each root of $P$ is an attracting fixed point of $f$, and the point at infinity is
%a repelling fixed point of $f$. The degree $d$ coincides with the number of
%distinct roots of $P$. Since the relation with the root-finding problem, the study of Newton maps became one of the major themes
%with general interest, both in discrete dynamical system (pure mathematics), and in root-finding algorithm (applied mathematics), see for example \cite{AR,HSS,Ro,RS,RWY,Tan,WYZ}.

\noindent{\bf Theorem B} (D-S and R-Y-Z).  \emph{If two  Newton maps are
combinatorially equivalent, and
\begin{enumerate}
\item  either they are both non-renormalizable,
\item or they are both renormalizable, and there is a bijection between their domains of
renormalization that respects quasi-conformal equivalence between the little Julia sets as well
as their combinatorial position.
\end{enumerate}
Then  they are quasi-conformally conjugate in a neighborhood of
the Julia set, and the domain of this quasiconformal conjugation can be chosen to include all Fatou components not in the basin of roots.
Furthermore,  invariant line fields of a Newton map only possibly exist at the little Julia sets coming from renormalization}.

Roughly speaking, two Newton maps are combinatorial equivalent means  all the
components of  their basins of roots  are connected to each other in the same way (see Section \ref{sec:combinatorial} for the precise definition); and the dynamics on the complement  of the basins of roots, if non-trivial,  is considered as the renormalization part, which essentially consists of finite hybrid embedded polynomials (see Section \ref{sec:puzzle-renormalization}). Therefore, the combinatorial equivalence gives no information of the renormalization part. It implies, to obtain a global rigidity, one has to add a local rigidity on the renormalization part, as stated in Theorem B.(2).

A critical points of a Newton map $f_p$ is called \emph{free} if it is not iterated to the basins of roots of $p$; and called \emph{renormalizable} if it lies in a little filled-in Julia set coming from renormalization (see Section \ref{sec:puzzle-renormalization}). It is clear that renormalizable critical points are always free. Combining Theorem A and B, it is easy to get a rigidity theorem for real (coefficients) Newton maps with only real renormalizable critical points.

\begin{proposition}\label{pro:rigidity}
 If two real Newton maps with only real renormalizable critical points are combinatorial equivalent, have a bijection between their parabolic points and have the same kneading sequences, then they are quasi-conformally conjugate in a neighborhood of
the Julia set, and the domain of this quasiconformal conjugation can be chosen to include all Fatou components not in the basin of  roots. Furthermore, any Newton map with only real renormalizable critical points carries no invariant line fields on the Julia sets.
\end{proposition}

For the definition of kneading sequences of real Newton maps, see Section \ref{sec:combinatorial}. Using this rigidity result, and following the idea in \cite{KSS1} (deducing the density of hyperbolicity from rigidity), we can prove the density of hyperbolicity of real Newton maps with all free critical points real.

\begin{theorem}[main]\label{thm:main}
Any real Newton map with all free critical points real can be approximated by hyperbolic real Newton maps of the same degree with all free critical points real.
\end{theorem}

\noindent {\bf The paper is organized as follows}: In Section \ref{sec:newton-graph}, we recall the definition of Newton graph given in \cite{MRS}; in Section \ref{sec:combinatorial}, we prove Proposition \ref{pro:rigidity}, and in Section \ref{sec:density}, we prove Theorem \ref{thm:main}.

{\noindent \bf Acknowledgement.} We thank Weixiao Shen to point out a useful reference. The author is supported by NSFC grant no. 11871354.

\section{Newton graphs}\label{sec:newton-graph}

\subsection{Basic results of Newton maps}\label{sec:basic}
In general, let $p(z)$ be a complex polynomial, factored as
\[p(z)=(z-a_1)^{n_1}\cdots(z-a_d)^{n_d},\]
 and $a_1,\ldots,a_d$ ($d\geq3$) are distinct roots of $p$, with multiplicities
$n_1,\ldots,n_d\geq1$, respectively.
Its Newton map $$f=f_p(z):=z-\frac{p(z)}{p'(z)}$$ has degree $d$ and fixes each root $a_i$ with multiplier $f'(a_i)=(n_i-1)/n_i$. Therefore, each root $a_i$ of $p$ corresponds to an attracting fixed point of $f$
with multiplier $1-1/n_i$. One may verify that $\infty$ is a repelling fixed point of $f$ with multiplier $d/(d-1)$. This discussion shows that a degree $d$ Newton map has $d+1$
distinct fixed points with specific multipliers. On the other hand, a well-known theorem of Head states that the fixed points together with the specific
multipliers can determine a unique Newton map:

\begin{proposition}[Head]\label{pro:head}
A rational map $f$ of degree $d\geq3$ is a Newton
map  if and only if f has $d+1$ distinct fixed points $a_1,\ldots,a_d,\infty$,  such that for each fixed point $a_i$, the multiplier takes the form $1-1/n_i$ with $n_i\in\N, 1\leq i\leq d$.
\end{proposition}

According to Shishikura \cite{Sh}, the Julia set of a Newton map is always
connected, or equivalently, all Fatou components are simply connected.

\subsection{Newton graphs}
Let $p(z)$ be a polynomial with $d$ distinct roots
 $a_1,\ldots,a_d$, and $f:=f_p$ be its Newton map. Denote $B_1,\ldots, B_d$ the immediate basins of the attracting fixed points $a_1,\ldots,a_d$ of $f$. We call the union of $B_1,\ldots,B_d$ the \emph{immediate roots-basin} of $f$, and
 \begin{equation}\label{eq:root-basin}
 \BBB_f:=\bigcup_{n\geq0} f^{-n}(B_1\cup\cdots\cup B_d)
 \end{equation}
 the \emph{roots-basin} of $f$.
 The map $f$ is called \emph{roots-basin \pf} if every critical point in $\BBB_f$ is eventually mapped to one of $a_1,\ldots,a_d$.

Let $f$ be a Newton map of degree $d$. By a standard quasi-conformal surgery on $\BBB_f$, we can obtain a unique Newton maps $f_*$ (up to affine conjugate), such that $f_*$ is roots-basin postcritically-finite, and $f,f_*$ are quasi-conformally conjugate on neighborhoods of their Julia sets together with the Fatou components not in their roots-basins
 (See \cite[Section 4]{DS} for a more detailed discussion). We call $f_*$ the \emph{roots-basin \pf represent} of $f$.

In general, let $R$ be a rational map and $\UUU$ the union of the grand orbits of several Fatou components of $R$. We say $R$ is \emph{\pf in $\UUU$} if each component of $\UUU$ contains at most one postcritical point. One can give a natural dynamical parameterization of all components of $\UUU$ (see \cite{Mil}).
\begin{lemma}\label{lem:coordinate}
There exist, so-called B\"{o}ttcher coordinates, $\{(U,\phi_U)\}_{U\in{\rm Comp}(\UUU)}$, such that
\begin{enumerate}
\item each $\phi_U:U\to\D$ is conformal;
\item $\phi_{R(U)}\circ R\circ \phi_U^{-1}(z)=z^{d_U}, z\in\D$, where $d_U:={\rm deg}(R|_U)$.
\end{enumerate}
\end{lemma}
For each $U\in{\rm Comp}(\UUU)$, the B\"{o}ttcher map $\phi_U$ is generally not unique. Once the B\"{o}ttcher coordinates $\{(U,\phi_U)\}_{U\in{\rm Comp}(\UUU)}$ are chosen, for any component $U$ of $\UUU$, we call the point $\phi_U^{-1}(0)$ the \emph{center} of $U$, and the preimages by $\phi_U$ of radii of $\D$ the \emph{internal rays (of $R$) in $U$}. By Lemma \ref{lem:coordinate}, the map $R$ sends an internal ray of $U$ homeomorphically onto an internal ray of $R(U)$. According to Douady-Hubbard's theory, if an internal ray in $U$ is eventually periodic by $R$, it \emph{lands} at $\partial U$, i.e.,  it accumulates at exact one point in $\partial U$.

The study of dynamics of Newton maps relies on an invariant graph constructed in \cite{MRS}. Let $f$ be a roots-basin \pf Newton map of degree $d$, with immediate roots-basin $B_1,\ldots,B_d$. The fixed internal rays in each $B_i$ land at fixed points in $\partial B_i$, which can only possible be $\infty$. Thus all fixed internal rays have a common landing point. We denote $\De_0$ the union of all fixed internal rays in $B_1,\ldots,B_d$ together with $\infty$.
Clear $f(\De_0)=\De_0$. For any $n\geq0$, denote by $\De_n$ the connected component of $f^{-n}(\De_0)$ that contains $\infty$. Following \cite{MRS}, we call $\De_n$  the \emph{Newton graph} of $f$ at level $n$. The vertex set $V_{\De_n}$ of $\De_n$ consists of iterated preimages of fixed points of $f$ contained in $\De_n$.
%It was proved in \cite[Theorem 3.4]{MLS} that there exists positive integer $N$ such that $\De_N$ contains all poles of $f$. This result implies immediately the following.

\begin{lemma}[\cite{MRS},Theorem 3.4]\label{lem:connected}
There exists $N\geq0$ such that the Newton graph  $\De_N$ contains all poles of $f$. Then $\De_{n+1}=f^{-1}(\De_n)$ and $\De_n\subseteq\De_{n+1}$ for any $n\geq N$.
\end{lemma}

\begin{definition}[canonical Newton graph]\label{def:newton-graph}
Let $N$ be the minimal number such that $\De_N$ contains all poles of $f$ and all critical points iterated to fixed points of $f$. We call $\De_f:=\De_N$ the \emph{canonical Newton graph} of $f$.
\end{definition}

%By this lemma, the Newton graphs naturally induce a puzzle as follows. Let $N$ be the minimal number such that $\De_N$ contains all poles of $f$. For each $k\geq0$, let $\WWW_k$ denote the collection of all connected components of $\ov{\C}\setminus \De_{N+k}$. We call $\WWW=\cup_{k\geq0}\WWW_k$ the \emph{$\De$-puzzle of $f$,} or  \emph{$\De(f)$-puzzle}, and call each element of $\WWW_k$  a \emph{$\De$-puzzle piece (of $f$) at level $k$}. By Lemma \ref{lem:connected}, we know that
% \begin{enumerate}
% \item each $\De$-puzzle piece is a simply-connected domain;
% \item two distinct $\De$-puzzle pieces are either disjoint or nested;
% \item each $\De$-puzzle piece of level $k+1$ ($k\geq0$) is contained in a $\De$-puzzle piece of level $k$;
% \item each $\De$-puzzle piece of level $k+1$ ($k\geq0$) is a connected component of the preimage by $f$ of a $\De$-puzzle piece of level $k$.
%\end{enumerate}

%Let us say a little more about the unbounded puzzle pieces, which will play an important role in our argument below. For each $k\geq0$, denote
%\[Y_k(\infty):=\bigcup\{\ov{P}:P\in \WWW_k,\infty\in\ov{P}\}.\]
%The following lemma is proved in \cite[Proposition 5.3]{WYZ}.

%\begin{lemma}\label{lem:infinity}
%The diameter of the intersection of $Y_k(\infty)$ and $\JJJ(f)$ converges to $0$ as $k$ tends to $\infty$.
%\end{lemma}

\subsection{Renormalization of Newton maps} \label{sec:puzzle-renormalization}
%Let $\WWW$ be the puzzle induced by Newton graphs defined in last section.
According to \cite{DH2}, A \emph{polynomial-like} map of degree $d\geq2$ is a triple $(g,U,V)$
where $U,V$ are topological disks in $\C$ with $\ov{U}\subseteq V$,
and $g:U\to V$ is a holomorphic proper map of degree $d$. The \emph{filled-Julia
set} of $g$ is the set of points in $U$ that never leave $V$ under iteration of $g$, i.e.,
$$K_g:=\bigcap_{n\geq0}g^{-n}(V),$$
and its \emph{Julia set} is defined as $J_g:=\partial K_g$.

Two polynomial-like maps $f$ and $g$ are \emph{hybrid equivalent} if there is a quasiconformal conjugacy $\psi$ between $f$ and $g$ that is defined on a neighborhood
of their respective filled-in Julia sets so that $\partial \ov{\psi}=0$ on $K_f$.
The crucial relation between polynomial-like maps and polynomials is
explained in the following theorem, due to Douady and Hubbard \cite{DH2}.

\begin{theorem}[Straightening Theorem]\label{thm:straigtening}
Let $f:U\to V$ be a polynomial-like map of degree $d\geq2$. Then $f$ is hybrid equivalent to a polynomial $P$ of the same degree. Moreover, if $K_f$ is connected, then $P$ is unique up to affine conjugation.
\end{theorem}

Now, let $f$ be a roots-basin \pf Newton map, and $W,W'$ be two complementary components of Newton graphs $\De,\De'$ respectively.
We call $\rho=(f^p,W,W')$ a \emph{renormalization  triple (of period $p$)} if:
\begin{enumerate}
\item $f^i(W),0\leq i\leq p-1,$ are pairwise disjoint (which are complementary components of Newton graphs $f^i(\De),0\leq i\leq p-1$ respectively), $f^p(W)=W'\supseteq W$, and ${\rm deg}(f^p|_W)\geq 2$;
\item the \emph{filled-in Julia set} of $\rho$, defined as
\[K_\rho:=\{z\in W\mid f^n(z)\in W\text{ for all $n\geq0$}\}\]
is connected.
\end{enumerate}

A roots-basin \pf Newton map is called \emph{renormalizable} if such defined renormalization triples exist; and \emph{no-renormalizable} otherwise. In the renormalizable case,
we call  $K_\rho$ and $J_\rho:=\partial K_\rho$ a \emph{little filled-in Julia set} and \emph{little Julia set} of $f$ (coming from renormalization). A general Newton map $f$ is called \emph{renormalizable/no-renormalizable} if its roots-basin \pf represent $f_*$ is renormalizable/no-renormalizable; and in the renormalizable case, the \emph{little (filled-in) Julia sets} of $f$ is the image of the little (filled-in) Julia sets of $f_*$ by the quasi-conformal conjugacy between $f$ and $f_*$ on $\ov{\C}\setminus \BBB_f$ and $\ov{\C}\setminus \BBB_{f_*}$ respectively.

Note that $f^p:W\to W'$ is not necessarily a polynomial-like map , because it is possible that either $\ov{W}\not\subseteq W'$ (although $W\subseteq W'$) or $W,W'$ are not disks (although they are simply-connected). However, the following result shows that one can revise $f^p:W\to W'$ to a polynomial-like map.

%by a modification of $W,W'$, one can get a renormalization triple of $f$ in the sense of Section \ref{sec:renormalization}.

\begin{proposition}[\cite{LMS1}, Lemma 4.19]\label{pro:puzzle-renormalization}
Let $(f^p,W,W')$ be a renormalization triple defined above. Then there exists a pair of Jordan domain $U\subseteq V$, such that $\ov{V}\subseteq W$, $f:U\to V$ is a polynomial-like map with  degree equal to ${\rm deg}(f^p|_W)$, and the filled-in Julia set of $f^p:U\to V$ equals to that of $(f^p,W,W')$.
\end{proposition}

Suppose that $\rho=(f^p,W,W')$ is a renormalization triple of $f$, and let $f^p:U\to V$ be the polynomial-like map in Proposition \ref{pro:puzzle-renormalization}. By Straightening Theorem, there exists a polynomial  (unique up to affine conjugacy) hybrid equivalent to $f^p:U\to V$. We also say that this polynomial is \emph{hybrid equivalent to $\rho$}.
Notice that, by definition, the filled-in Julia set of any renormalization  triple is disjoint from the Newton graphs of all levels. So the filled-in Julia sets of different renormalization triples either coincides or are disjoint.

\section{Combinatorial rigidity of real Newton maps}\label{sec:combinatorial}

The following definition comes from \cite{DS}.
\begin{definition}[combinatorial equivalent]\label{def:equivalent}
 Two roots-basin \pf Newton maps $f,\wt{f}$ of the same degree are called \emph{combinatorial equivalent} if
\begin{enumerate}
\item their canonical Newon graphs (see Definition \ref{def:newton-graph}) $\De,\wt{\De}$ have the same level, and they are homeomorphic and topologically conjugate, respecting vertices;
\item there is a bijection between critical points of $f$ on $\ov{\C}\setminus \De$ and of $\wt{f}$ on $\ov{\C}\setminus \wt{\De}$
that respects degrees and itineraries with respect to (complementary components
of) $\De,\wt{\De}$.
\end{enumerate}
 Two Newton maps are \emph{combinatorially equivalent} if their corresponding roots-basin \pf represents are combinatorially equivalent.
 \end{definition}

 Now we define the \emph{kneading sequence} for a real Newton map.  Let $f$ be a real Newton map, and $c_1<\cdots<c_k$ are all free critical points of $f$ (critical points not in the roots-basin) in $\R$. If no such critical points exist, the kneading sequence is defined as $\emptyset$. Otherwise, these  points divides $\R$ into $k+1$ open intervals $L_1,\ldots,L_{k+1}$ labeled in the positive order. For each $c_i$, we define the \emph{kneading sequence of $c_i$} by $\ell(c_i)=\epsilon_0\epsilon_1\cdots\epsilon_n\cdots$ such that $\epsilon_n:=j$ if $f^n(c_i)\in L_j$ and $\epsilon_n:=j_*$ if $f^n(c_i)=c_j$. We call $(\ell(c_1),\ldots,\ell(c_k))$ the \emph{kneading sequence} of $f$.

\begin{proof}[Proof of Proposition \ref{pro:rigidity}]
Note that the quasi-conformal surgery on the roots-basins do not change the kneading sequences of real Newton maps. We can then assume $f$ and $\wt{f}$ are both roots-basin \pf. By the combinatorial equivalence of $f$ and $\wt{f}$, for each large $n$, there exists a homeomorphism between their Newton graph  $\De_n$ and $\wt{\De}_n$ of level $n$ on which $f$ and $\wt{f}$ are topologically conjugate. This family of homeomorphisms induces a bijection between the complementary components of $\De_n$ and $\wt{\De}_n$ for any large $n$, and this bijection is compatible with the dynamics.

Let $\rho=(f^k,W,W')$ be a renormalization triple of $f$. By the bijection between the complementary components of Newton graphs of $f$ and $\wt{f}$, we get a corresponding renormalization triple $\wt{\rho}=(\wt{f}^k,\wt{W}.\wt{W}')$, where $\wt{W},\wt{W}'$  correspond to $W,W'$ respectively. Let $P$ and $\wt{P}$ denotes the polynomials hybrid equivalent to $\rho$ and $\rho'$ respectively. Since $f$ and $\wt{f}$ have only real renormalizable critical points, then $P$ and $\wt{P}$ are real polynomials of the same degree with all critical points real. As $f,\wt{f}$ are combinatorial equivalent and have the same kneading sequences, there is a bijection between the critical points of $P,\wt{P}$, also keeping the orders. The same kneading sequences of $f$ and $\wt{f}$ still implies the same kneading sequences of $P$ and $\wt{P}$.  By a standard kneading theory (see e.g. \cite{MT}), the real polynomials $P,\wt{P}$ are topologically conjugate on $\R$. Since the parabolic periodic points between $f$ and $\wt{f}$ are supposed to have a one-to-one correspondence, the conjugacy on $\R$ of $P$ and $\wt{P}$ is a bijection on their parabolic periodic points. Thus, applying Theorem A, the polynomials $P$ and $\wt{P}$ are quasi-conformally conjugate on $\C$.

Using the argument above to every renormalization triple, we get that $f$ and $\wt{f}$ satisfy the properties of Theorem B, and hence quasi-conformally conjugate on a neighborhood of the union of Julia set and all Fatou components not in the roots-basin.

Let $f$ be a real Newton map with all renormalizable critical points real.  Then each polynomial hybrid equivalent to a renormalization triple of $f$ is real with only real critical points. By \cite[Theorem 1.4]{CvST}, this polynomial has no invariant line fields on the  Julia set. It then follows from Theorem B that $f$ has no invariant line fields on its Julia set.
\end{proof}

\section{Density of hyperbolicity of real Newton maps}\label{sec:density}

\subsection{Perturbation of Newton graphs}
Let $f$ be a rational map of degree $d\geq2$. Suppose that  $z$ is an attracting/repelling periodic point of $f$.
The implicit function theorem implies immediately the following result.
\begin{lemma}\label{lem:continuation}
There exists a small neighborhood $V$ of $f$ (in the space of rational maps of degree $d$) and a holomorphic map $\zeta_z:V\to \C$ such that $\zeta_z(f)=z$ and $\zeta_z(g)$ is the unique attracting/repelling periodic points $g$ near $z$ with the same  period as those of $z$.
\end{lemma}

%Let  $U(f)$ a (supper-)attracting periodic Fatou component of $f$ with an attracting periodic point $a\in U(f)$. Then the point $a$ belongs to a unique periodic attracting domain $U(g)$ for all rational map $g$ closed enough to $f$, which we call the \emph{deformation of $U(f)$} at $g$. Similarly, if $\UUU(f)$ denotes the union of several periodic attracting domain of $f$, then $\UUU(g)$ is defined as the union of the corresponding deformation at $g$ of the components of $\UUU(f)$.

Let $\UUU(f)$ be the union of grand orbits of some periodic attracting domains $U_1(f)\ldots,U_k(f)$ of $f$. For each $i$, the attracting periodic point $\xi_i\in U_i(f)$ belongs to a unique periodic attracting domain $U_i(g)$ of $g$ for all $g$ close to $f$. We denote by $\UUU(g)$ the union of grand orbits  by $g$ of $U_1(g),\ldots,U_k(g)$. Recall that $g$ is \pf in $\UUU(g)$ if every component of $\UUU(g)$ contains at most one postcritical point of $g$.
%collection of Fatou components of $f$ contained in the grand orbit of $\UUU(f)$. We call the map $f$  \emph{\pf in $GO(\UUU)$} if each $U\in GO(\UUU(f))$ contains at most one postcritical points of $f$, i.e., ${\rm Post}_f\cap\UUU(f)$ is a finite set.

\begin{lemma}\label{lem:perturbation-rational}
Let $f$ be a rational map, with $\UUU(f),\ \UUU(g)$ defined as above.  Assume that $\Lambda$ is a connected set containing $f$ such that all $g\in\Lambda$ are \pf in $\UUU(g)$. Let $U(f)$ be any component of $\UUU(f)$ with the center $x$.
\begin{enumerate}
 \item There exists a continuous map $x(g)$ on $\Lambda$ near $f$, such that $x(f)=x$,  ${\rm deg}(f|_{x})={\rm deg}(g|_{x(g)})$, and $x(g)$ is the center of the unique component $U(g)$ of $\UUU(g)$ which has the same period and preperiod as  $U(f)$. We call $U(g)$ the \emph{deformation of $U(f)$ at $g$}.
      \item Let $I(g)$ be a preperiodic internal ray of $g$ in $U(g)$ with fixed preperiod and period, where $U(g)$ is the deformation of $U(f)$ at $g$. If the landing point $z(g)$ of $I(g)$ converge to $z$, and  $z$ is eventually repelling for $f$, then
          $$I:=\limsup_{\Lambda\ni g\to f}I(g):=\{w\mid \exists: \Lambda\ni g_n\to f,w_n\in I(g_n);\ {\rm s.t}\ w_n\to w\}$$
          is the internal ray of $f$ in $U(f)$ landing at $z$.
          \end{enumerate}
\end{lemma}
\begin{proof}
For (1), we first assume $x$ is periodic. Let $x(g):=\xi_x(g)$ be the continuation of $x$ in Lemma \ref{lem:continuation}, and chose $U(g)$ as the Fatou component of $g$ containing $x(g)$. Clearly $U(g)$ is a component of $\UUU(g)$. By the \pf of $g$ in $\UUU(g)$, the point $x(g)$ is the center of $U(g)$.
Let us now deal with the preperiodic case by induction. Let us assume that $y:=f(x)$, the function and $y(g)$ satisfies the properties of the lemma: we need to find $x(g)$ such that
$x(g)\to x$, $g(x(g))=y(g)$ and ${\rm deg}(f|_x)=\text{deg}(g|_{x{g}})$ as $\Lambda\ni g\to f$.

Let $U(g)$ be the unique component of $g$ containing $x$, which is easily seen a component of $\UUU(g)$. We define $x(g)$ to be the center of $U(g)$. Since $y(g)\to y=f(x)$ and $U(g)$ contains $x$, then, for $g$ close to $f$, we have $g(U(g))$ contains $y(g)$. It follows that $g(x(g))=y(g)$.
Set $\delta:={\rm deg}(f|_x)$. By Rouch\'{e}'s theorem, any given small neighborhood of $x$  contains exactly $\delta$ preimages by $g$ of $y(g)$ (counting with multiplicity) for $g$ close to $f$. Note that all these preimages belong to $U(g)$, and are the centers of some Fatou components of $g$. So these preimages must coincide with $x(g)$. The proof of point (1) is then completed.

To see point (2), if $I$ is period, the conclusion holds by Goldberg and Milnor's proof in \cite[Appendix B]{GM}. By induction, it suffices to prove $\limsup_{g\to f} I(g)=I$ under the assumption that $\limsup_{g\to f} g(I(g))=f(I)$. As $g\to f$, we can choose  B\"{o}ttcher coordinates $\varphi_f$ of $U(f)$ and $\varphi_g$ of $U(g)$ such  that $\varphi_g^{-1}:\D\to U(g)$ converge
uniformly on compact sets to $\varphi_f^{-1}:\D\to U(f)$. It follows that $I':=U(f)\cap\limsup_{g\to f} I(g)$ is an internal ray of $U(f)$. On the other hand, note that $\partial U(f)\cap\limsup_{g\to f} I(g)$ is compact, connected and contains the point $z$, and the map $f$ sends $\partial U(f)\cap\limsup I(g)$ into the set $\partial f(U(f))\cap\limsup g(I(g))$, which is by induction a singleton. Then we get $\partial U(f)\cap\limsup_{g\to f} I(g) =\{z\}$, and hence $I'=I$.
\end{proof}

\begin{proposition}[perturbation of Newton graph]\label{pro:perturbation-Newton-graph}
Let $f$ be a roots-basin \pf Newton map of degree $d$. Let $\Lambda\ni f$ be a connected set contained in a small neighborhood of $f$ (in the space of Newton maps of degree $d$), such that all maps $g\in\Lambda$ are roots-basin \pf. Let $\De$ be the Newton graph of $f$ of level $k$. Suppose that, for each Julia critical points $c\in\De$, there exists a continuous map ${\bf c}:\Lambda\to\C$ with ${\bf c}(f)=c$, such that \begin{enumerate}
\item ${\rm deg}(g|_{{\bf c}(g)})={\rm deg}(f|_c)$ for all $g\in\Lambda$;
\item if $c_i,c_j$ are two Julia critical points in $\De$ with $f^k(c_i)=c_j$, then $g^k({\bf c}_i(g))={\bf c}_j(g)$ for all $g\in \Lambda$;
\item if $c$ is a Julia critical point in $\De$ with $f^k(c_i)=\infty$, then $g^k({\bf}c(g))=\infty$ for all $g\in \Lambda$.
\end{enumerate}
Then, we have a continuous map $h:\Lambda\times \De\to \ov{\C}$ such that $h_g:\De\to \ov{\C}$ is injective with the image $h_g(\De)$ equal to the Newton graph of $g$ of level $k$, and the equality $h_g\circ f=g\circ h_g$ holds on $\De$.
\end{proposition}
\begin{proof}
Let $V_\De$ denote the vertex set of $\De$, consisting of the iterated preimages of the fixed points of $f$ contained in $\De$.  Since all maps $g\in\Lambda$ are roots-basin postcritically-finite, and the Julia points in $V_\De$ are eventually repelling, by Lemma \ref{lem:continuation}, Lemma \ref{lem:perturbation-rational}.(1), and the assumption of the proposition,  we get a continuous function $h:\Lambda\times V_\De\to \ov{\C}$, such that $h_f=id|_{V_\De}$ and the equality $h_g\circ f=g\circ h_g$ holds on $V_\De$.

Let $e$ be an edge of $\De$. Then it is an internal ray in a component $U$ of the roots-basin, with one endpoint, say $x$, the center of $U$, and the other endpoint, denoted by $y$, on the boundary of $U$. By Lemma \ref{lem:perturbation-rational} and the assumption (1), the point $h_g(x)$ is the center of the deformation $U(g)$ of $U$, and $h_g(y)$ belongs to $\partial U(g)$. Thus, we can
continuously extends $h_g$ to $\De$, such that for each edge $e$ of $\De$, with the notations above,  let $h_g$ map $e$ homeomorphically onto the internal ray in $U(g)$ with endpoints $h_g(x),h_g(y)$. Then $h_g:\De\to \ov{\C}$ is clearly an injective. We may further require $h_g\circ f=g\circ h_g$ on $\De$.
To see this, one can first extends $h_g$ on $\De_0$ (Newton graph of $f$ at level 0) using the B\"{o}ttcher coordinates, such that $h_g\circ f=g\circ h_g$ holds on $\De_0$, and then lift $h_g$ by the dynamics of $f$ and $g$ to $\De$.
By Lemma \ref{lem:perturbation-rational} again, such constructed map $h:\Lambda\times \De\to\ov{\C}$ is continuous.

To finish the proof of this proposition, it remains to show $h_g(\De)$ is the Newton graph $\De_k(g)$ of $g$ at level $k$. By the argument above, we have known that $h_g(\De)$ is a sub-graph of $\De_k(g)$.  Suppose the conclusion is false. Then, we can find a sequence $\{g_n\in\Lambda,n\geq1\}$ converging to $f$,
 a vertex $v$ of $\De$, and a vertex $u_n$ of $\De_k(g)$ not in $h_g(\De)$, such that $v_n:=h_{g_n}(v)$ and $u_n$ are the endpoints of an edge $e_n$ of $\De_k(g_n)$. By taking subsequences if necessary, Lemma \ref{lem:perturbation-rational} shows that the edge $e_n$ converge to an internal ray $e$ of $f$ which joins $v$ and $u=\lim_{n\to\infty} u_n$. Since $u_n\in V_{\De_k(g_n)}$, then $g_n^k(u_n)$ is a fixed point of $g_n$. It follows that $f^k(u)$ is a fixed point of $f$. Combining the fact that $\De\cup e$ is connected, we get  $u\in \De$, and hence $u_n\in h_{g_n}(\De)$, a contradiction to the choice of $u_n$.
\end{proof}

\subsection{Normalized form of Newton maps}
Note that, since Newton maps have the fixed point $\infty$, two Newton maps are conformal conjugate if and only if they are affine conjugate.
Consider the family of Newton maps
\[\FFF_d:=\{f_p\mid \text{$p$ has degree $d$ and $d$ distinct roots}.\}\]
If $f_p\in\FFF_p$, the roots of $p$ are supper-attracting fixed points of $f_p$ (see Subsection \ref{sec:basic}). Moreover, any general Newton map is quasi-conformally conjugate to a map in $\FFF_d$ outside their immediate basins of roots.

Let $f_p$ and $f_{\wt{p}}$ be two maps in $\FFF_d$.
It is easy to see that: there exists an affine map $\g$ such that $f_{\wt{p}}\circ\g=\g\circ f_p$ if and only $\g$ sends the roots of $p$ onto the roots of $\wt{p}$. So we can use the normalized property of $f_p$ as follows:
\[\emph{the sum of the roots of $p$ is $0$ and the product of the roots of $p$ is $1$}.\]
The space of normalized Newton maps can be written as
$$\NNN_d:=\{f_p\mid p(z)=z^d+a_{d-2}z^{d-2}+\cdots+a_1z+1,\ \text{with $d$ pairwise distinct roots}\}.$$
It is the complement in $\C^{d-2}$ of the discriminant of $p$.

Notice that any Newton map $f\in\FFF_d$ is conformal equivalent to a $f_p$ with $p$ a monic, centered polynomial of degree $d$, and such $f_p$ is conformal equivalent to a normal form if and only if its roots are all non-zero. Hence, if $p$ has zero root, it can not be written  a normal form under conformal conjugacy, but it can be approximated by the maps in $\NNN_d$.

In the real case, we denote
$$\NNN_d(\R):=\{f\in \NNN_d\mid \text{the coefficients of $f$ are real}\}.$$
Then $\NNN_d(\R)$ is the complement in $\R^{d-2}$ of the discriminant of the real polynomials in the normal form. Let $\YYY$ denote the subfamily of $\NNN_d(\R)$ consisting of maps with all free critical points (critical points not in the roots-basin) real.
We endow $\YYY$ the topology of subspace in $\R^{d-2}$, i.e., a sequence $\{f_{p_n},n\geq1\}$ converges to $f_p$ in $\YYY$ if the coefficients of $p_n$ converge to those of $p$. %Note that each connected component of $\YYY$ is path-connected.
By the discussion above, to prove Theorem \ref{thm:main}, it is enough to show that the hyperbolic maps in $\YYY$ are dense.

\subsection{The proof of Theorem \ref{thm:main}}
%Let $f\in \NNN_d(\R)$. By Ma\~{n}\`{e}-Sad-Sullivan, the map $f$ is called \emph{J-stable} (in $\NNN_d(\R)$) if there exists a neighborhood $\Lambda\subseteq \NNN_d(\R)$ of $f$ such that, for any $g\in \Lambda$, there exists a homeomorphism $h_g:\JJJ_f\to \JJJ_g$ such that $h_g\circ f=g\circ h_g$ on $\JJJ_f$, and $\JJJ_g$ varies continuously (in the Hausdorff topology) with $g\in \Lambda$.
%Let $\NNN_d^{s}(\R)$ denotes the subset of $\NNN_d(\R)$ which consists of all J-stable parameters of $\NNN_d(\R)$. By \cite[Theorem B]{MSS}, the space $\NNN_d^s(\R)$ is dense, open in $\NNN_d(\R)$.

%For any  map $f\in \NNN_d(\R)$, it has no irrational indifferent periodic points since the multiplicity of any periodic point is real. Moreover, if $f$
%has a parabolic periodic point $z_0$ of order $k$, one can find  a map $g\in \NNN_d(\R)$ near $f$ such that $z_0$  splits into $k$ simple periodic point of $g$. It follows from \cite[Theorem]{MSS} that $f\not\in \NNN_d^s(\R)$. Therefore,
%\begin{equation}\label{eq:only-attracting}
%\emph{the Fatou set of any map in $\NNN_d^s(\R)$ consists of attracting domains.}
%\end{equation}

%Now, we define a subset $\XXX\subseteq \NNN_d^s(\R)$ such that all Julia critical points are simple for any map in $\XXX$. Then $\XXX$ is the complementary of finitely many hypersurfaces in $\NNN_d^s(\R)$, and hence still a dense open subset of $\NNN_d(\R)$. So, to prove Theorem \ref{thm:main}, it is enough to show that every component of $\XXX$ contains a hyperbolic map.

For every $f\in\YYY$, we define $\tau(f)$ as the number of critical points of $f$ which are contained in
the basins of the  attracting cycles (not only roots-basin). Note that the function $\tau:\YYY\to \N\cup\{0\}$ is lower semi-continuous.
Let
\[\XXX:=\{f\in\YYY:\text{$\tau$ is locally maximal at $f$}\}.\]
As $\tau$ is uniformly bounded above, the set $\XXX$ is dense in $\YYY$. Moveover, from the lower semi-continuity of $\tau$, the map $\tau$ is constant in a neighborhood of any $f\in\XXX$. Thus $\XXX$ is open, dense in $\YYY$, and $\tau$ is constant on any connected component $X$ of $\XXX$, which we denote by $\tau(X)$.
Thus, to prove Theorem \ref{thm:main}, we just need to show $\tau(X)=2d-2$ for every  component $X$ of $\XXX$.

On the contrary, let $X$ be a component of $\XXX$ with $\tau(X)=:m<2d-2$. 
%We claim that the maps in $X$ with the following property are dense and open:
%\begin{equation}\label{eq:simple-critical-point}
%\emph{the  critical points not in the attracting basins are simple (local degree $2$)}.
%\end{equation}
%The openness is obvious, so we just  show the density of these maps. Choose any $f_p\in X$. Let $q$ be a normalized real polynomials of degree $d$ such that, the roots of $p''(z)=0$ lying in the attracting domains of $p$ are also roots of $q''(z)=0$ with the same multiplicities, and the other roots of $q''(z)$ are simple and real. We can choose such $q$ arbitrarily close to $p$. It follows that all free critical points of $f_q$ are real, hence  $f_q\in X$, and Property \ref{eq:simple-critical-point} holds for $f_q$.  The claim is then proved.
By a standard quasi-conformal deformation, the map $f$ can be continuously deformed in $X$ to a map which is \pf in its attracting basins (not only the roots-basin). Without loss of generality, we assume the initial $f$ is \pf in its attracting basins. %Let $V$ be an open neighborhood of $f$ in $X$. 

We write the critical points of $f$ as $c_1,\ldots,c_{r_1};c_{r_1+1},\ldots,c_r;c_{r+1},\ldots,c_{r+s}$ with multiplicities $\mu_1,\ldots,\mu_{r_1};\mu_{r_1+1},\ldots,\mu_r;\mu_{r+1},\ldots,\mu_{r+s}$ respectively, such that 
 \begin{enumerate}
 \item $c_1,\ldots,c_{r_1}$ lie in the roots-basin of $f$, with $c_1,\ldots,c_d$ the supper-attracting fixed points;
 \item $c_{r_1+1},\ldots,c_r$ are free critical points in some attracting basins other than the roots-basin;
 \item $c_{r+1},\ldots,c_{r+s}$ are the remaining free critical points (not in the attracting basins).
 \end{enumerate}
Then $c_{r_1+1},\ldots,c_{r+s}$ are real, and the numbers $m,r,s$ satisfy $m=\mu_1+\ldots+\mu_r$ and $\mu_1+\cdots+\mu_{r+s}=2d-2$.

%Let $V$ be a sufficiently small open neighborhood of $f$ in $\NNN_d(\R)$ (recall that $\NNN_d(\R)$ is an open set in $\R^{d-2}$). Then there exist continuous function
%\[{ c}_{r+i}:V\to\R,\quad 1\leq i\leq s,\]
%such that ${\bf c}_{r+i}(f)=c_{r+i}$ and ${\bf c}_{r+i}(g)$ are critical points of $g$.

Consider all the real rational maps $g$ (not only Newton maps) of degree $d$ near $f$ such that
\begin{enumerate}
\item $c_1(g),\ldots,c_{r+s}(g)$ are critical points of $g$ with multiplicities  $\mu_1,\ldots,\mu_{r+s}$ respectively;
\item $g(\infty)=\infty, c_1(g)+\ldots+c_d(g)=0$ and $c_1(g)\cdots c_d(g)=1$ (normalized property).
\end{enumerate}
Then the component $Z$ of such maps containing $f$ is a sub-manifold of dimension $r+s$ embedded in the space of real rational maps of degree $d$ (see, for example, \cite[Theorem 2.1]{LSS}). Notice also that $c_i(g)$ is real for $i=r_1+1,\ldots,r+s$.

For each $i=1,\ldots,r$, there exists a unique number $k_i\in\{1,\ldots,r\}$ and a number $q_i$ such that $f^{q_i}(c_i)=c_{k_i}$ and $f^{j}(c_i)$ are not critical points for $j=1,\ldots,q_i-1$. In fact, the triples $\{(i,k_i,q_i),i=1,\ldots,r\}$ records the critical relation of $f$ in the attracting basins. This critical relation induces $r$ codimension $1$ subvarieties  of $Z$, defined by the equations
\[g^{q_i}(c_i(g))=c_{k_i}(g),\ {\rm with}\ g\in Z {\rm\ and\ }\ i=1,\ldots,r.\]
Let $W$ denote the component  of the intersection of the $r$ subvarieties of $Z$ which contains $f$. Following \cite[Theorem 3.2]{LSS}, locally near $f$, the set $W$ is a sub-manifold  of dimension $s$ embedded in $Z$. This critical relation and the normalized property also implies that all maps in $W$ near $f$ are Newton maps in $\NNN_d(\R)$. Notice that, when $g\in W$ is close to $f$, its critical points $c_1(g),\ldots,c_{r_1}(g)$ belong to the roots-basin, and the remaining critical points $c_{r_1+1}(g),\ldots,c_{r+s}(g)$ are real. It means the free critical points of $g$ are real, i.e., $g\in X$. Hence, we may assume $W\subseteq X$. %where $V$ is a chosen open neighborhood of $f$ in $X$.

Since the number of critical points in the attracting basins is constant in $W$, then any $g\in W$
 is \pf in its attracting basins, with $c_1(g),\ldots,c_{r_1}(g)$ in the roots-basin and $c_{r_1+1}(g),\ldots,c_r(g)$ in the attracting basins other than the roots-basin. Note that the maps in $W$ have no irrational indifferent periodic points (because the free critical points are real), then the critical points $c_{r+1}(g),\ldots,c_{r+s}(g)$ are contained in either the parabolic basins or the Julia set. By Proposition \ref{pro:rigidity}, the maps in $W$ carry no invariant line fields on the Julia sets. Therefore, Theorem 6.9 in \cite{McS} is reduced to  the following form in our case.

 \begin{lemma}\label{lem:dimension}
 For any $g\in W$, let $QC(g)$ denote the component containing $g$ of the set $\wt{g}\in \NNN_d$ which are quasi-conformally conjugate to $g$. Then $QC(g)$ is an embedded complex orbiford in $\NNN_d$ with (complex) dimension $b_{cp}-b_p$, where $b_{cp}$ denotes the number of grand orbit of critical points of $g$ contained in the parabolic basins, and $b_p$ denotes the number of parabolic cycles.
 \end{lemma}

 \begin{proof}[Completion of the proof of Theorem \ref{thm:main}]
 For positive integers $n$ and $1\leq i,j\leq s$, consider
\[M_{n,i,j}=\{g\in W:g^n(c_{r+i}(g))=c_{r+j}(g)\}\text{ and }M_{n,i,\infty}=\{g\in W:g^n(c_{r+i}(g))=\infty\}.\]
Then each of these $M_{n,i,j}$ and $M_{n,i,\infty}$ is a subvariety of dimension at most $s-1$. We claim that at least one of these subvarieties has dimension $s-1$.

On the contrary, the set $W-\bigcup_{n,i,j}(M_{n,i,j}\cup M_{n,i,\infty})$ is arc-connected by the following fact (\cite[Fact 2.1]{KSS1}). \\
Fact. \emph{Let $m$ be a positive number and $B$ a Euclidean ball in
$\R^m$. Let $M_i, i=1,2,\ldots,$ be embedded real analytic sub-manifolds of $B$ such that $dim(M_i)\leq m-2$. Then $B-\bigcup_{i=1}^\infty M_i$ is arc-connected.}

By shrinking $W$ if necessary, for any $g,\wt{g}$ in this complementary set, we have
\begin{itemize}
\item their canonical Newton graphs have the same level, and $g,\wt{g}$ conjugate on their canonical Newton graphs (by Lemma \ref{lem:perturbation-rational});
\item their is a bijection between their critical points which keeps the orders: this is duo to the construction of $W$; 
\item their corresponding critical points have the same itineraries with respect to their canonical Newton graphs;
\item they have the same kneading sequences;
\item there is a bijection between their parabolic periodic points: to see this, by the first three properties, we get a bijection between their renormalization filled-in Julia sets $K_{i,g}$ and $K_{i,\wt{g}}$; and by the fourth property, there is a topological conjugacy between $g^{k_i}$ and $\wt{g}^{k_i}$ on $I_{i,g}:=K_{i,g}\cap \R$ and $I_{i,\wt{g}}:=K_{i,\wt{g}}\cap \R$ respectively, where $k_i$ denotes the renormalization period; since $\tau(X)$ is constant in $X$,  this conjugacy must send parabolic points of $g$ to those of $\wt{g}$.
\end{itemize}
It then follows from Proposition \ref{pro:rigidity} that all maps in $W-(M_{n,i,j}\bigcup_{n,i,j} M_{n,i,\infty})$ are quasi-conformally conjugate on the neighborhoods of their Julia sets together with the Fatou components not in the roots-basins. Since the maps in $W$ are \pf in their attracting basins,  these quasi-conformal conjugacies can be extended to the sphere.
Hence $$W\subseteq QC_{\tiny{\R}}(\wt{f})\cup(\bigcup_{n,i,j}(M_{n,i,j}\cup M_{n,i,\infty}))$$ with some $\wt{f}\subseteq W$, where $QC_{\tiny{\R}}(\wt{f}):=QC(\wt{f})\cap \NNN_d(\R)$.  By Lemma \ref{lem:dimension}, $QC_{\tiny{\R}}(\wt{f})$ has (real) dimension at most $s-1$, so $W$ is a countable union of sub-manifolds of codimension at least one,
impossible.
The proof of the claim is then completed.

Therefore, we obtain a real analytic sub-manifold $W_1$ embedded in $W$ of dimension $s-1$, on which one of the critical relation
\[g^{q_1}(c_{r+i_1}(g))=c_{r+j_1}(g) \text{ or\ \ } g^{q_1}(c_{r+i_1}(g))=\infty\]
is persisting.
 As previous, for positive integers $n$ and $1\leq i,j\leq s$, consider
\[M_{n,i,j}^{(1)}:=M_{n,i,j}\cap W_1\quad\text{and}\quad M_{n,i,\infty}^{(1)}=M_{n,i,\infty}\cap W_1.\]
For a map $g\in W_1$, if this persisting critical relation happens in the Julia set, the critical points contained in the parabolic domain is at most $s-1$; otherwise, the number of grand orbits of critical points in the parabolic domain is at most $s-1$.
Then, using the same argument as that in the claim above, we get a sub-manifold $W_2$ embedded in $W_1$ of dimension $s-2$, on which two critical relations are persisting. Inductively, for every $1\leq k\leq s$, we get a sub-manifold $W_k\subseteq W$ of dimension $s-k$ such $k$ distinct critical relations are persisting in $W_k$. Finally, when $k=s$, we get a  sub-manifold $W_s\subseteq W$ of $0$-dimension, i.e., a map $g\in W$. Note that all critical points $c_{r+1}(g),\ldots,c_{r+s}(g)$ must be iterated to $\infty$ because the persisting relations contains no cycles (otherwise $\tau(X)$ is not constant). So $g$ is \pf.

In \cite{AR,AR2}, the author proved that any Misiurewicz rational maps (including \pf ones) can be approximated by hyperbolic rational maps. Using  their argument, we will find a map near $g$ in $W$ with at leas one more critical points in the attracting basins than $g$. It contradicts $\tau(X)$ is constant, completing the proof of Theorem \ref{thm:main}.

By \cite[Theorem 3.2]{LSS}, the maps in $W$ near $g$ can be parameterized by $B_0(1):=\{v=(v_1,\ldots,v_s)\in\R^s:\sum_{i=1}^s|v_i|<1\}$, and $g_0=g$.  When $\R$ intersects the Fatou set of $g$, let $I\subseteq \R\cap\FFF_g$ be a closed interval. Then there exists $\delta_0>0$ such that $I$ is contained in the attracting basin of $g_{v}$ for all $v\in B_0(\delta_0)$.  It is proved in \cite{AR2} that, for any $0<\delta<\delta_0$, there exists large $n=n_\delta$, such that the set $\{g_v^n(c_1(v))\mid v\in B_0(\de)\}$ intersects $I$. This means that there exists a $v_*\in B_0(\de)$ such that $g_{v_*}^n(c_1(v_*))\subseteq \FFF_{g_*}$, and hence $\tau(g_*)> \tau(g)$.
If $\R\subseteq \JJJ_g$,
it is proved in \cite{AR} that, for any $\de>0$, there exists a large $n$, such that the set $\{g_{v}^n(c_1(v))\mid v\in B_0(\de)\}$ covers the set $\{c_1(v)\mid v\in B_0(\de)\}$. It follows that there exists a $v_*\in B_0(\de)$ such that $g_{v_*}^n(c_1(v_*))=c_1(v_*)$. Hence $g_{v_*}$ is a map in $W$ with $\tau(g)>\tau(g)$.
\end{proof}

\vspace{1cm}

\noindent Yan Gao, \\
Mathemaitcal School  of Sichuan University, Chengdu 610064,
P. R. China. \\
Email: gyan@scu.edu.cn

\end{document}